\documentclass[12pt,a4paper,reqno]{amsart}
\usepackage[cp1251]{inputenc}
\usepackage[ukrainian,english]{babel}
\usepackage{latexsym,amsfonts,amsthm,amssymb,amsmath}
\usepackage{graphicx,graphics,hhline}
\usepackage{euscript}
\usepackage[all]{xy}
\usepackage [right=2cm, left=2cm, top=2cm, bottom=2cm]{geometry}

\theoremstyle{plain}
\newtheorem{theorem}{Theorem}

\newtheorem{lemma}{Lemma}
\newtheorem{corollary}{Corollary}
\theoremstyle{definition}
\newtheorem{definition}{Definition}
\newtheorem{example}{Example}

\begin{document}
\author{Tetiana M. Osipchuk}

\address{Institute of Mathematics of the National Academy of Sciences of Ukraine, Tereshchenkivska str. 3, UA-01004, Kyiv, Ukraine}
\email{osipchuk@imath.kiev.ua}

\title[Topological properties of closed weakly $m$-convex sets]{Topological properties \\ of closed weakly $m$-convex sets}

\keywords{convex set, closed set,  $m$-convex set, weakly $m$-convex set, $m$-nonconvexity point of set, real Euclidean space}
\subjclass{32F17, 52A30}

\maketitle
\small
\noindent {\bf Abstract.} The present work considers the properties of generally convex sets in the $n$-dimensional real Euclidean space $\mathbb{R}^n$, $n>1$,
known as weakly $m$-convex, $m=1,2,\ldots,n-1$.  An open set of
$\mathbb{R}^n$ is called \textbf{\emph{weakly $m$-convex}} if for any boundary point of the set
there exists an $m$-dimensional plane passing through this point and not intersecting the given set. A
closed set of $\mathbb{R}^n$ is called \textbf{\emph{weakly $m$-convex}} if it is approximated from the
outside by a family of open weakly $m$-convex sets. A point of the complement of a set of $\mathbb{R}^n$ to the whole space is called an \textbf{\emph{$m$-nonconvexity point}} of the set  if any $m$-dimensional plane passing through the point intersects the set. It is proved that any closed, weakly $(n-1)$-convex set in $\mathbb{R}^n$ with non-empty set of $(n-1)$-nonconvexity points consists of not less than three connected components. It is also proved that the interior of a closed, weakly $1$-convex set with a finite number of components in the plane is weakly $1$-convex. Weakly $m$-convex domains and closed connected sets  in $\mathbb{R}^n$ with non-empty set of $m$-nonconvexity points are constructed for any $n\ge 3$ and any $m=1,2,\ldots,n-2$.
\large
\section{Introduction}

 As is well known, a set of the multidimensional real Euclidean space $\mathbb{R}^n$ is called \emph{convex} if, together with its two arbitrary points, it contains the entire segment connecting the points \cite{Leh}. Moreover, the intersection of an arbitrary number of convex sets is again a convex set.
The intersection of all  convex sets containing a given set  $X\subset\mathbb{R}^n$ is called the
\emph{convex hull} of the set $X$ \cite{Leh} and is denoted by
$$
\mathrm{conv}\,X=\bigcap\limits_{K\supset X}K, \quad\mbox{where sets $K$ are convex.}
$$

Consider some generalizations of the convexity notion.

We recall that any $m$-dimensional affine subspace of the space  $\mathbb{R}^n$, $1\le m< n$, is called an \emph{$m$-dimensional plane} \cite{Roz}.

\begin{definition}\label{def1} (\cite{Zel}) A set $E\subset\mathbb{R}^n$ is called {\bf\emph{ $m$-convex with respect to a point}} $x\in \mathbb{R}^n\setminus E$,  $1\le m<n$, if there exists
an $m$-dimensional plane $L$ such that $x\in L$ and $L\cap E=\varnothing$.
\end{definition}

\begin{definition}\label{def2} (\cite{Zel}) A set $E\subset\mathbb{R}^n$ is called {\bf\emph{$m$-convex}},  $1\le m<n$, if it is $m$-convex with respect to every point $x\in \mathbb{R}^n\setminus E$.
\end{definition}
The intersection of an arbitrary number of $m$-convex sets is again an $m$-convex set \cite{Stef1}. On the other hand, there exist convex sets that are not $m$-convex. An open triangle in the plane together with one of its vertices $x$ is convex. However, one can not draw the straight line through a point inside the triangle's sides with common vertex $x$ which does not intersect the set.  Thus, the set is not $1$-convex. And conversely, there exist $m$-convex sets that are not convex. The example of such a set is the union of two open or closed non-overlapping rectangles in the plane which are symmetric with respect  to the axis $Ox$.

The properties of $m$-convex compacts in the space $\mathbb{R}^n$, related to estimating cohomology groups, are investigated by Yuri~B.~Zelinskii in \cite{Zel}. The p  roperties of $(n-1)$-convex sets in $\mathbb{R}^n$ are considered by V.~L.~Melnyk \cite{Mel} and, under some additional conditions, by A.~I.~Gerasin \cite{Ger1}, \cite{Ger2}. In particular, the topological classification of $(n-1)$-convex sets in $\mathbb{R}^n$, $n\ge 2$, with smooth boundary is obtained in \cite{Mel} and is the following: {\it any $(n-1)$-convex set in $\mathbb{R}^n$, $n\ge 2$, with smooth boundary is convex, or consists of no more than two unbounded connected components, or is given by the Cartesian product $E^1\times \mathbb{R}^{n-1}$, where $E^1$ is a subset of $\mathbb{R}$}.

We shall use the following standard notations. For a set $G\subset\mathbb{R}^n$ let $\overline{G}$
be its closure, $\mathrm{Int}\, G$ be its interior, and $\partial
G=\overline{G}\setminus\mathrm{Int}\, G$ be its boundary.

\begin{definition}\label{def3} (\cite{Zel01}) An open set  $G\subset\mathbb{R}^n$ is called \textbf{\emph{weakly $m$-convex}},  $1\le m<n$, if it is $m$-convex with respect to any point  $x\in\partial G$.
\end{definition}
\begin{definition}\label{def6} (\cite{Aiz3})
They say that a set $A$ \textbf{\emph{is approximated from the outside}} by a family of open sets $A_k$,
$k=1,2,\ldots$, if $\overline{A}_{k+1}$ is contained in $A_k$, and $A=\cap_kA_k$.
\end{definition}
It can be proved that any set approximated from the outside by a
family of open sets is closed.
\begin{definition}\label{def4} \textup{(\cite{Zel01}, \cite{Dak})} A
closed set $E\subset\mathbb{R}^n$ is called \textbf{\emph{weakly $m$-convex}} if it can be approximated
from the outside by a family of open weakly $m$-convex sets.
\end{definition}

Thus, any weakly $m$-convex set  $A$ is either open or closed. Among closed weakly
$m$-convex sets there are also sets with empty interior:
$$
A=\overline{A}=\overline{A}\setminus \mathrm{Int}\, A=\partial A.
$$

Any weakly $m$-convex set  is obviously weakly $p$-convex, $p<m$.
It is also easy to see that any open convex set $E\subset\mathbb{R}^n$ is weakly $m$-convex, $1\le m<n$. Indeed, for any boundary point of a convex set there is a supporting hyperplane of the set \cite{Leh}. And, since $E$ is open, the supporting hyperplane does not intersect $E$. Thus, $E$ is $(n-1)$-convex and, therefore, $m$-convex, $1\le m<n$.
Similarly,  any closed convex set $\mathbb{R}^n$ is weakly $m$-convex, $1\le m<n$, since it can be approximated from the outside by a family of open convex sets homothetic to the given one. Moreover, there are open and closed weakly $m$-convex sets which are not convex.

The geometric and topological properties of weakly $m$-convex sets are investigated in \cite{Dak01}. In particular, the following proposition is proved in \cite{Dak01}: {\it If a set $E_1$ is weakly $m$-convex and a set $E_2$ is weakly $p$-convex, $p\le m$, then the set $E_1\cap E_2$ is weakly $p$-convex.} The properties of the class of generalized convex sets on Grassmannian manifolds which are closely related to the properties of the conjugate sets (see Definition 2, \cite{Zel01}) are investigated in \cite{Zel01}. This class includes $m$-convex and weakly $m$-convex sets in $\mathbb{R}^n$.

\begin{definition}\label{def5}\textup{(\cite{Osi})}
A point $x\in \mathbb{R}^n\setminus E$ is called an \textbf{\emph{$m$-nonconvexity point of a set $E\subset\mathbb{R}^n$}} if any $m$-dimensional plane passing through $x$ intersects $E$.
\end{definition}

The set of all $m$-nonconvexity points of a set  $E\subset\mathbb{R}^n$, $n\ge 2$, is denoted by $(E)_m^{\triangle}$, $1\le m<n$. Thus, if a set $E\subset\mathbb{R}^n$ is not $m$-semiconvex, then obviously $(E)_m^{\triangle}\ne\varnothing$. And let $$(E)_1^{\triangle}:=(E)^{\triangle},\quad E\subset \mathbb{R}^n,\quad n\ge 2.$$

Let us denote the classes of $m$-convex and weakly $m$-convex sets in $\mathbb{R}^n$, $n\ge 2$, $1\le m<n$, by $\mathbf{C^n_m}$ and $\mathbf{{WC}^n_m}$, respectively.  Any open set of the class $\mathbf{C^n_m}$ clearly belongs to the class $\mathbf{{WC}^n_m}$. The converse statement is not true. It turns out that the class  $\mathbf{WC^n_m}\setminus \mathbf{C^n_m}$, $n\ge 2$, of open weakly $m$-convex but not $m$-convex sets is not empty for any $m=1,2,\ldots,n-1$ \cite{Osi}, \cite{Dak}.
 Moreover, the following proposition is true:

\begin{lemma}\label{theor01} \textup{(\cite{Dak})}
An open set of the class $\mathbf{WC^n_{n-1}}\setminus \mathbf{C^n_{n-1}}$ consists of not less than three connected components.
\end{lemma}

The estimate of the number of components of the sets of the class $\mathbf{WC^n_m}\setminus \mathbf{C^n_m}$, $n\ge 3$, $1\le m<n-1$, is different, which proves the following

 \begin{lemma}\label{theor02} \textup{(\cite{Osi})}
There exist domains in the space $\mathbb{R}^n$, $n\ge 3$, of the class $\mathbf{WC^n_m}\setminus \mathbf{C^n_m}$, $1\le m<n-1$.
\end{lemma}

The following two lemmas are also proved in \cite{Osi} and will be used for the proof of the main results of this paper.
\begin{lemma}\label{pro4}\textup{(\cite{Osi})}
Let a closed set $E\subset\mathbb{R}^n$, $n\ge 2$, belong to the class $\mathbf{WC^n_m}\setminus \mathbf{C^n_m}$, $1\le m<n$. Then for any family of open, weakly
$m$-convex sets $E^k$, $k=1,2,\ldots$, approximating the set $E$ from the outside, there exists an index
$k_0\in\mathbb{N}$ such that every set $E^k$, $k=k_0,k_0+1,\ldots$, of the family is not $m$-convex.
\end{lemma}

\begin{lemma}\label{pro3}\textup{(\cite{Osi})}
Let $E^p\subset\mathbb{R}^p$, $p\ge 2$, be an open or a closed set of the class $\mathbf{WC^p_1}\setminus \mathbf{C^p_1}$. Then the set $E:=E^p\times\mathbb{R}^{n-p}\subset\mathbb{R}^n$, $n\ge3$, belongs to the class $\mathbf{WC^n_{n-p+1}}\setminus \mathbf{C^n_{n-p+1}}$.
\end{lemma}
The examples of open and closed sets of the class $\mathbf{WC^n_{n-1}}\setminus \mathbf{C^n_{n-1}}$ with three and more connected components are constructed in \cite{Osi} (see Examples 1--4, \cite{Osi}). It is also proved in \cite{Osi} that the compact sets of the class $\mathbf{WC^n_{n-1}}\setminus \mathbf{C^n_{n-1}}$ consist of not less than three connected components.

The present work proceeds  the research of Yu. Zelinskii  and his students by investigating the topological properties mainly of closed sets of the classes  $\mathbf{WC^n_m}\setminus \mathbf{C^n_m}$, $n\ge 2$, $1\le m<n$. In particular, answers are given to some of the questions posed in \cite{Osi}. Namely,  in chapter 2 it is proved that not only compact set but any closed set of the class $\mathbf{WC^n_{n-1}}\setminus \mathbf{C^n_{n-1}}$ consists of not less than three connected components. It is also proved that the interior of a closed, weakly $1$-convex set with a finite number of components in the plane is weakly $1$-convex. In chapter 3 domains and closed connected sets of the classes $\mathbf{WC^n_m}\setminus \mathbf{C^n_m}$, $n\ge 3$, $1\le m<n-1$, are constructed.

\section{Topological properties of closed sets of the class $\mathbf{WC^n_m}\setminus \mathbf{C^n_m}$, $n\ge 2$, $1\le m< n$}
First, give some denotations.  The interval between points $x,y\in\mathbb{R}^n$ will be written as
$xy$ and the distance between the points  will be written as $|x-y|$. Let $U(y):=\{x\in\mathbb{R}^n:|x-y|<\varepsilon\}$, $\varepsilon>0$, be a neighborhood of a point
$y\in\mathbb{R}^n$.


\begin{lemma}\label{lemm2}
Let a closed, weakly $m$-convex set $E\subset\mathbb{R}^n$, $n\ge 2$, $1\le m<n$, with the number of components $N$ be given. Then $E$ is approximated from the outside by a family of open, weakly $m$-convex sets $E^k$, $k=1,2,\ldots$, such that the number of components of each set $E^k$ is not greater than $N$.
\end{lemma}
\begin{proof}[Proof.]  Since $E$ is  weakly $m$-convex, there exists a family of open weakly $m$-convex sets $G^k$, $k=1,2,\ldots$, approximating  $E$ from the outside.
Let every set $E^k$, $k=1,2,\ldots$, consist only of the components of $G^k$ containing points of $E$. Consider a point $y_k\in \partial E^k$. Then $y_k\in \partial G^k$. Since $G^k$ is open and weakly $m$-convex, there exists an $m$-dimensional plane $L_{y_k}$ passing through ${y_k}$ and such that $L_{y_k}\cap G^k\ne \varnothing$. Since $G^k\supset E^k$, then $L_{y_k}\cap E^k\ne \varnothing$. Thus, any set  $E^k$, $k=1,2,\ldots$,  is open, weakly $m$-convex, and consists of components the number of which is not greater than $N$.

Since $G^k\supset \overline{G^{k+1}} \supset \overline{E^{k+1}}$ and $\overline{E^{k+1}}$ is contained only in those components of $G^k$ which contain points of $E$, then $E^k\supset \overline{E^{k+1}}$. Let us prove that $E=\cap_kE_k$.

Suppose $x\in \cap_kE_k$, then $x\in E_k$ for any $k=1,2,\ldots$. Since $E_k\subset G_k$, then $x\in G_k$ for any $k=1,2,\ldots$. Therefore, $x\in\cap_kG_k=E$. Now let $x\in E$. Since $G_k\supset E$, $k=1,2,\ldots$, the point $x$ belongs to some component $G^0_k$ of $G_k$ for any $k=1,2,\ldots$.  Then $x\in G^0_k\subset E_k$, $k=1,2,\ldots$, which gives $x\in \cap_kE_k$.

Thus, $E$ is approximated  from the outside by the family of open sets $E_k$, $k=1,2,\ldots$, by Definition \ref{def6}.
\end{proof}

\begin{theorem}\label{corol6}
Let a closed set $E\subset\mathbb{R}^n$, $n\ge 2$, of the class $\mathbf{WC^n_m}\setminus \mathbf{C^n_m}$, $1\le m<n$, with the number of components $N$ be given. Then $E$ is approximated from the outside by a family of open sets $E^k$, $k=1,2,\ldots$, of the class $\mathbf{WC^n_m}\setminus \mathbf{C^n_m}$ such that the number of components of each set $E^k$ is not greater than $N$.
\end{theorem}
\begin{proof}[Proof.]
By Lemma \ref{lemm2},  $E$ is approximated from the outside by a family of open, weakly $m$-convex sets $G^k$, $k=1,2,\ldots$, such that the number of components of each set $G^k$ is not greater than $N$. By Lemma \ref{pro4}, there exists an index $k_0$ such that every set $G^k$, $k=k_0,k_0+1,\ldots$, belongs to the class $\mathbf{WC^n_m}\setminus \mathbf{C^n_m}$. Thus, $E$ is approximated from the outside by the family of sets
$$
E^k:=G^{k_0+(k-1)},\quad k=1,2,\ldots,
$$
satisfying the theorem conditions.
\end{proof}

\begin{theorem}\label{theortheor}
Let a closed set $E\subset \mathbb{R}^n$ belong to the class $\mathbf{WC^n_{n-1}}\setminus \mathbf{C^n_{n-1}}$. Then $E$ consists of
not less than three components.
\end{theorem}

\begin{proof}[Proof.]

Suppose $E$ is connected. Then, by Theorem \ref{corol6},  it can be approximated from the
outside by a family of domains $E^k$, $k=1,2,\ldots$, of the class $\mathbf{WC^n_{n-1}}\setminus \mathbf{C^n_{n-1}}$.  But this contradicts Lemma \ref{theor01}. Thus, $E$ is disconnected.

Suppose $E$ consists of two components.  By Theorem \ref{corol6},  it can be approximated from the outside by a family of open sets $E^k$, $k=1,2,\ldots$, of the class $\mathbf{WC^n_{n-1}}\setminus \mathbf{C^n_{n-1}}$ consisting
of one or two components. This contradicts Lemma \ref{theor01}. Thus, $E$
consists of more than two components. Examples 1--4 in \cite{Osi} complete the proof.
\end{proof}

\begin{definition}\label{def8}
The set of all points of the straight lines  passing through a point $x\in\mathbb{R}^n\setminus A$ and intersecting a set $A\subset\mathbb{R}^n$ is called the {\bf\emph{cone of the set $A$ with respect to the point $x$}} and is denoted by $\mathrm{C}_xA$. We suppose that $x\notin \mathrm{C}_xA$ whenever $A$ is open and $x\in\mathrm{C}_xA$ otherwise.
\end{definition}
\begin{definition}
The set of all points of the rays starting at a point $x\in\mathbb{R}^n\setminus A$ and passing through a set $A\subset\mathbb{R}^n$ is called the {\bf\emph{semicone of the set $A$ with respect to the point $x$}} and is denoted by $\mathrm{S}_xA$. We suppose that $x\notin \mathrm{S}_xA$ whenever $A$ is open and $x\in\mathrm{S}_xA$ otherwise.
\end{definition}

\begin{lemma} \label{lemm12}
Let a set $E\subset \mathbb{R}^2$ be open and convex and let $x\in \mathbb{R}^2\setminus E$. Then  $\mathrm{S}_xE$ is an open angle of value not greater than $\pi$.
\end{lemma}
\begin{proof}
Since $E$ is open and connected, $\mathrm{S}_xE$ is also open and connected. Therefore, $\mathrm{S}_xE$ is an open angle. Suppose its value is greater than $\pi$. Then there exists a line $\gamma(x)$ passing through $x$ and such that $\gamma(x)\subset \mathrm{S}_xE$. Let $\eta_x^1$, $\eta_x^2$ be the complementary rays starting at $x$ in $\gamma(x)$. There exist points $x^1\in E\cap\eta_x^1$, $x^2\in E\cap\eta_x^2$ by the definition of $\mathrm{S}_xE$.  Then the set $E$ is not convex, since  $x\in \overline{x^1x^2}$ and  $x\not\in E$. We have now reached a contradiction. Thus, the assumption is wrong, and the value of $\mathrm{S}_xE$ is not greater than $\pi$.
\end{proof}

\begin{corollary}\label{corollary5}
Let a set $E\subset \mathbb{R}^2$ be open and convex and $x\in \mathbb{R}^2\setminus E$. Then $\mathrm{C}_xE$ is the union of two vertical open angles of value $\le\pi$.
\end{corollary}

\begin{theorem}\label{theor4}
Let $E\subset \mathbb{R}^2$ be a closed set with a finite number of components and such that $\mathrm{Int}\, E\ne\varnothing$. If $E$ is weakly $1$-convex, then $\mathrm{Int}\, E$ is weakly $1$-convex.
\end{theorem}

\begin{proof}[Proof.]
Suppose $\mathrm{Int}\, E$ is not weakly $1$-convex. Then there exists a $1$-nonconvexity point $y\in\partial E$ of the set $\mathrm{Int}\, E$.

Suppose $E_i$, $i=1,\ldots,k$, are the components of $\mathrm{Int}\, E$. Let $\mathrm{C}_{i,j}=\mathrm{C}_yE_i\cap \mathrm{C}_yE_j$, $i,j=1,\ldots,k$ (see Figure \ref{Fig1} a)).  Since $y$ is a $1$-nonconvexity point of $\mathrm{Int}\, E$, then for any fixed index $i\in\{1,\ldots,k\}$ there exist the indices $j(i)\in\{1,\ldots,k\}$ such that $\mathrm{C}_{i,j(i)}\ne\varnothing$. Since the cones $\mathrm{C}_yE_i$, $i=1,\ldots,k$,  are open, we can reduce them so that the intersections of the reduced cones remain non empty. Denote the reduced cones  by  $\widetilde{\mathrm{C}}_{y}E_i$, $i=1,\ldots,k$. Then $\overline{\widetilde{\mathrm{C}}_{y}E_i}\subset \mathrm{C}_yE_i$. The boundary of  $\widetilde{\mathrm{C}}_{y}E_i$ consists of two straight lines passing through $y$, considering Corollary \ref{corollary5}. Denote them by $\gamma^1_i(y)$, $\gamma^2_i(y)$.  Moreover, $\gamma^1_i(y), \gamma^2_i(y)\subset \mathrm{C}_yE_i$. Thus,   $\gamma^1_i(y)\cap  E_i\ne\varnothing$, $\gamma^2_i(y)\cap  E_i\ne\varnothing$ by Definition~\ref{def8}. \begin{figure}[h]
	\centering
   \includegraphics[width=14 cm]{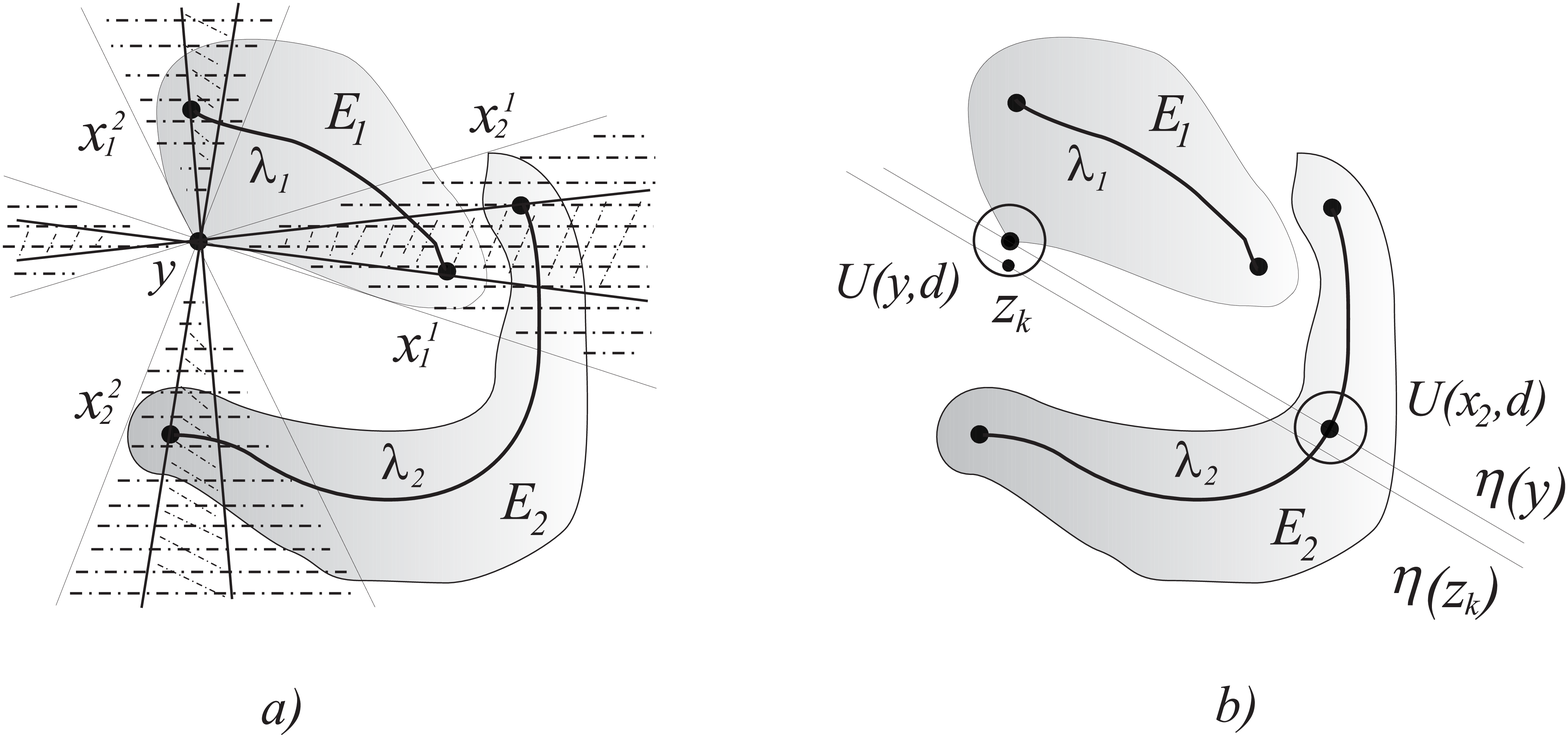}
	\caption{}\label{Fig1}
\end{figure}
Let
 $$x_i^1\in \gamma^1_i(y)\cap  E_i,\quad x_i^2\in \gamma^2_i(y)\cap  E_i,\,\,i=1,\ldots,k.$$

 Construct curves $\lambda_i\subset   E_i$, $i=\overline{1,k}$,  connecting the points $x_i^1$, $x_i^2$.
 Then for any straight line $\gamma(y)$ passing through the point $y$, there exists $i\in\{1,\ldots, k\}$ such that  $\gamma(y)\cap  \lambda_i\ne\varnothing$.

 Consider the function
 $$d_j(x)=\inf\limits_{x^0\in \partial E_j}|x-x^0|,\quad x\in E_j,\quad j=\overline{1,k}.$$
 It is continuous  in the domain $E_j$, $j=\overline{1,k}$. Then its restriction on the compact $\lambda_j$, $j=\overline{1,k}$, reaches its minimum $d_j>0$ on this compact, i. e.,
 $$
 d_j=\min\limits_{x\in \lambda_j}d_j(x),\quad j=\overline{1,k}.
 $$
 Since $E$ has the finite number of components, there exists
 $$
 d=\min\limits_{j=\overline{1,k}}d_j>0.
 $$
Then for any point $x\in \lambda_j$, $j=\overline{1,k}$, its neighborhood $U(x, d)\subset E$.
 Consider the neighborhood $U(y, d)$ of the point $y$ (see Figure \ref{Fig1} b)).
Since $E$ is weakly $1$-convex, there exists a family of open, weakly $1$-convex sets $G_k$, $k=1,2,\ldots$, approximating $E$ from the outside. This gives that starting from some index $k_0$,  $\partial G_k\cap U(y, d)\ne\varnothing$, $k\ge k_0$.  Let $z_k\in\partial G_k\cap U(y, d)$, $k=k_0, k_0+1,\ldots$. Draw an arbitrary straight $\eta(z_k)$ passing through $z_k$.  The straight $\eta(y)$ parallel to $\eta(z_k)$ and passing through $y$ intersects some curve $\lambda_q$, $q\in\{1,\ldots, k\}$ at a point $x_q$. Since $U(x_q, d)\subset E$ and $\eta(z_k)\cap U(x_q,d)\ne\varnothing$, then $\eta(z_k)\cap E\ne\varnothing$. Since $G_k\supset E$, $k=1,2,\ldots$, then $\eta(z_k)\cap G_k\ne\varnothing$, $k\ge k_0$.

 Since we choose the straight $\eta(z_k)$ arbitrarily, the point $z_k\in\partial G_k$ is a $1$-nonconvexity  point of $G_k$, $k=k_0,k_0+1,\ldots$, and we have now reached a contradiction. Thus, the assumption is wrong, and the theorem is proved.
\end{proof}
\begin{figure}[h]
    \centering
   \includegraphics[width=3 cm]{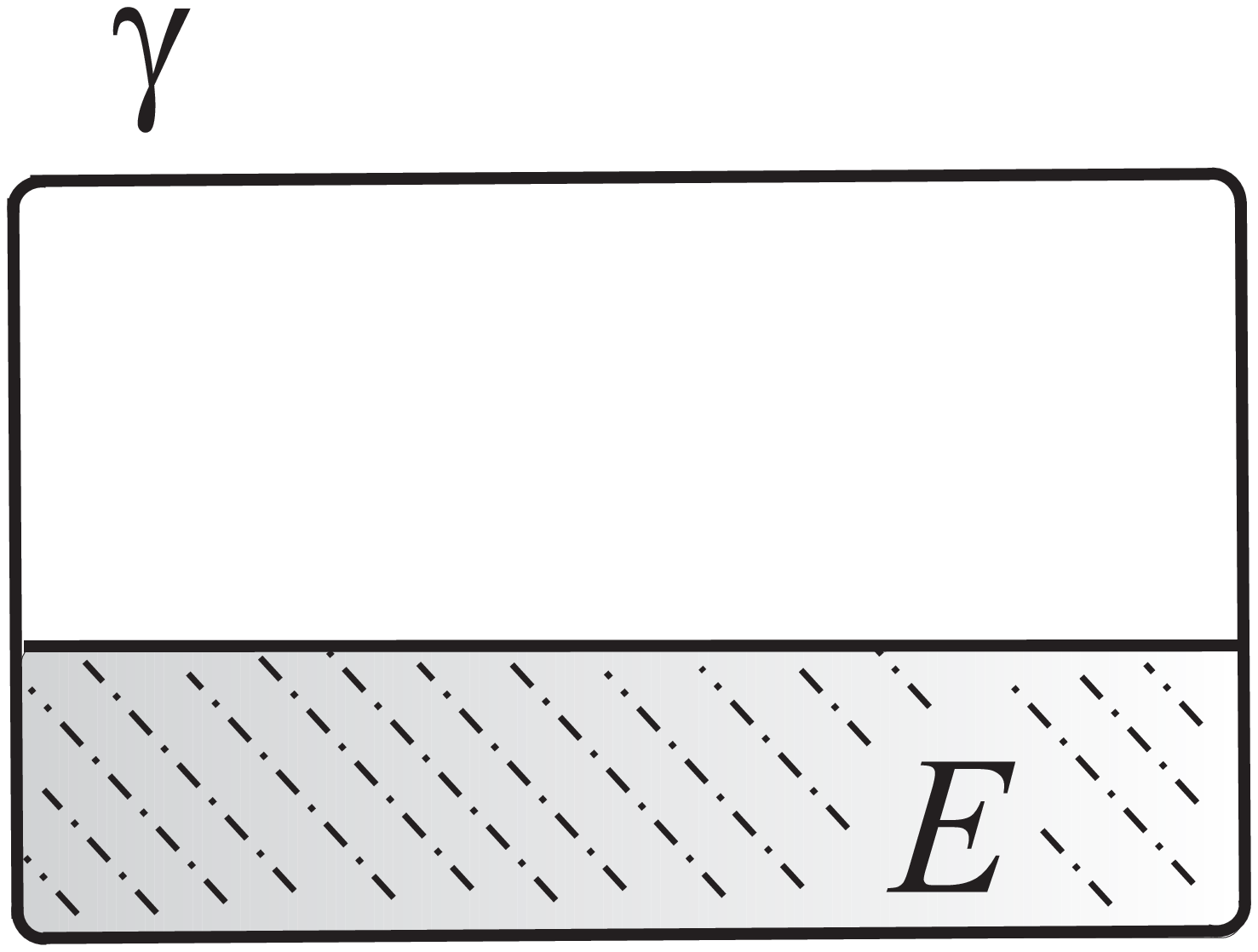}
    \caption{}\label{Fig2}
\end{figure}
The converse statement is not always true. An example of a closed, not weakly $1$-convex set such that its interior  is weakly $1$-convex is as follows. Consider an open convex set $E\subset\mathbb{R}^2$ such that $\mathrm{Int}\,\overline{ E}=E$ and connect any two of its boundary points by a curve $\gamma\subset \mathbb{R}^2\setminus \overline{E}$ (see Figure \ref{Fig2}). Then the closed set $\gamma\cup\overline{ E}$ is not weakly $1$-convex and its interior is weakly $1$-convex, since $\mathrm{Int}\,(\gamma\cup\overline{ E})=E$.

\section{Connected sets of the class $\mathbf{WC^n_m}\setminus \mathbf{C^n_m}$, $n\ge 3$, $1\le m<n-1$}

The estimate of the number of components of the closed sets of the class $\mathbf{WC^n_m}\setminus \mathbf{C^n_m}$, $n\ge 3$,  $1\le m<n-1$, is expectedly the same as for the open sets of this class. To prove this, first, provide here the folllowing
\begin{lemma}\label{lemm4}
The closure of an open set $E$ of the class $\mathbf{WC^n_m}\setminus \mathbf{C^n_m}$, $n\ge 2$, is not $m$-convex, $1\le m<n$.
\end{lemma}
\begin{proof}[Proof.]  Since  $E\in\mathbf{WC^n_m}\setminus \mathbf{C^n_m}$, there exists an $m$-nonconvexity point $x\in \mathbb{R}^n\setminus \overline{E}$ of the set $E$. Since $E\subset \overline{E}$, any $m$-dimensional plan passing through $x$ and intersecting $E$ intersects $\overline{E}$ as well. Thus, $x$ is an $m$-nonconvexity point of $\overline{E}$.
\end{proof}

In \cite{Osi} the examples of open and closed sets of the class $\mathbf{WC^n_{n-1}}\setminus \mathbf{C^n_{n-1}}$, $n\ge 2$,  were provided. Construct here a closed set of the class $\mathbf{WC^2_1}\setminus \mathbf{C^2_1}$ in a slightly different way.

\begin{example}\label{example4} Consider an open equilateral triangle $a_0b_0c_0$ and straight lines $\gamma_0^k$, $k=1,2,3$, containing the triangle sides. Let $\gamma^k_{t}$, $k=1,2,3$, $t\in(0,1]$, be the straight lines not intersecting $a_0b_0c_0$, parallel to the respective lines $\gamma_0^k$, and such that the distance between  $\gamma_0^k$ and  $\gamma^k_{t}$ equals $t$ (see Figure~\ref{Fig3}). Let $$a_{t}=\gamma^2_{t}\cap\gamma^3_{t}, b_{t}=\gamma^3_{t}\cap\gamma^1_{t}, \,\, c_{t}=\gamma^1_{t}\cap\gamma^2_{t}, \,\,t\in[0,1].$$
\begin{figure}[h]
    \centering
   \includegraphics[width=8 cm]{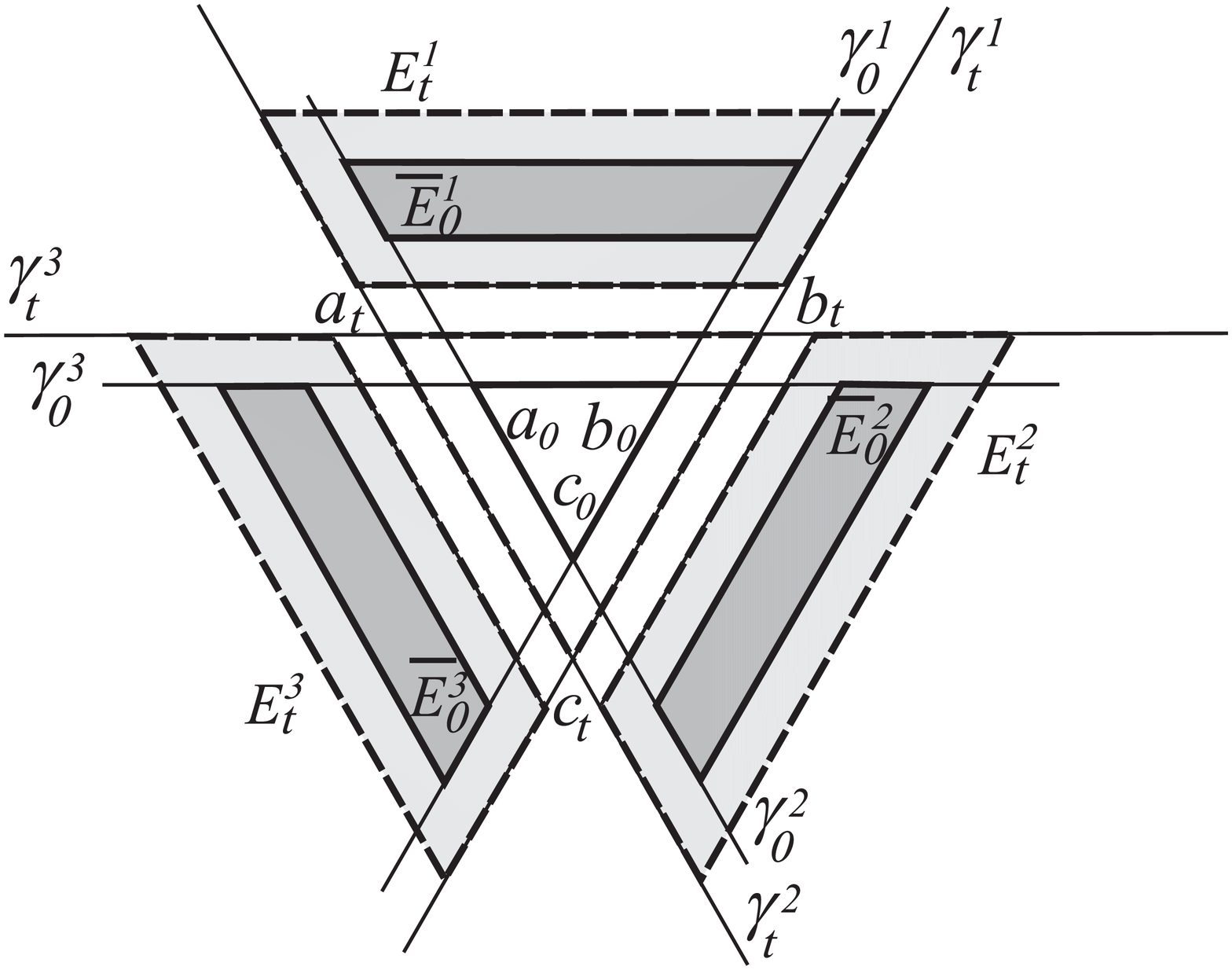}
    \caption{}\label{Fig3}
\end{figure}
Let $\alpha^1$ be the angle with vertex $a_{1}$, generated by  $\gamma^2_{1}$, $\gamma^3_{1}$, and containing the open triangle $a_{1}b_{1}c_{1}$. Let $\alpha^2$ be the angle with vertex $b_{1}$, generated by  $\gamma^3_{1}$, $\gamma^1_{1}$, and containing $a_{1}b_{1}c_{1}$. And let $\alpha^3$ be the angle with vertex $c_{1}$, generated by  $\gamma^1_{1}$, $\gamma^2_{1}$, and containing $a_{1}b_{1}c_{1}$.
Inscribe an open trapezium $E^k_{1}$, $k=1,2,3$, with height $>2$ into the angle $\alpha^k$ such that the parallel sides of $E^k_{1}$ are parallel to the straight $\gamma^k_{1}$ and
 $$\overline{E^k_{1}}\cap \overline{a_{1}b_{1}c_{1}}=\varnothing.$$

 Let $E^k_t\subset E^k_1$,  $t\in[0,1)$, $k=1,2,3$,  be the open trapezium the sides of which are parallel to the respective sides of $E^k_1$ and the distance between the sides of $E^k_t$ and the respective sides of $E^k_1$ equals $1-t$.
Then the open sets
$$
E_{t}=\bigcup\limits_{k=1}^3E^k_{t},\quad t\in[0,1],
$$
belong to the class $\mathbf{WC^2_1}\setminus \mathbf{C^2_1}$. Indeed, for a fixed $t\in[0,1]$, any point of $\partial E_{t}$ belongs to one of the sides of the trapeziums $E^k_{t}$, $k=1,2,3$, and the straight line passing through this side does not intersect $E_{t}$ by the construction.  Moreover, $(E_{t})^{\triangle}=a_{t}b_{t}c_{t}\ne\varnothing$, $t\in[0,1]$.

In addition, the closed set $\overline{E_{0}}$ belongs to the class $\mathbf{WC^2_1}\setminus \mathbf{C^2_1}$ by Lemma \ref{lemm4} and considering the fact that the family of sets $E_{1/k}$, $k=1,2,\ldots$, approximates $\overline{E_{0}}$ from the outside.
\end{example}

\begin{theorem}
There exist closed connected sets in the space $\mathbb{R}^n$, $n\ge 3$, of the class $\mathbf{WC^n_m}\setminus \mathbf{C^n_m}$, $1\le m<n-1$.
\end{theorem}
\begin{proof} [Proof.]  Prove the theorem by constructing examples of appropriate sets.
First construct the domains in the space
$\mathbb{R}^3$ of the class  $\mathbf{WC^3_1}\setminus \mathbf{C^3_1}$ approximating from the outside a closed connected set of the same class.
Consider the open sets
$$
E_{0},\quad E_k:=E_{1/k}, \quad k=1,2,\ldots,
$$
of the class $\mathbf{WC^2_1}\setminus \mathbf{C^2_1}$ constructed in Example \ref{example4}. By the construction,  $(E_0)^\triangle\subset(E_{k+1})^\triangle\subset (E_{k})^\triangle\subset (E_1)^\triangle=a_1b_1c_1$, $k=1,2,\ldots$. Further, for the convenience in the notations, we set
$$
1/0:=0.
$$
Consider the sets
\begin{equation}\label{theor5}
\tilde{\tilde{E}}_{k}^3:=E_{k}\times [1/k-s,s-1/k],\quad s>1, \quad k=0,1,2,\ldots.
\end{equation}
Let  $P_{k}^2\subset\mathbb{R}^2$ be the convex hull of the set $E_{k}$, $k=0,1,2,\ldots$. Construct the following prisms:
\begin{equation*}
\begin{split}
Pl^3_{k}:=P_{k}^2\times \left[-1/k-1-s, 1/k-s\right],&\\
Pr^3_{k}:=P_{k}^2\times
\left[s-1/k,s+1+1/k\right],&\quad  k=0,1,\ldots.
\end{split}
\end{equation*}
Now consider the sets
$$
\tilde{E}_{k}^3:=\mathrm{Int}\,(Pl^3_{k}\cup \tilde{\tilde{E}}^3_{k}\cup Pr^3_{k}),\quad k=0,1,\ldots.
$$
They are $1$-convex with respect to any point of $\partial \tilde{E}_{k}^3$ except the points of  the respective triangles:
\begin{equation*}
\begin{split}
\widetilde{Rl}^2_{k}:=\{(x_1,x_2,x_3)\in \partial \tilde{E}_{k}^3: (x_1,x_2)\in (E_k)^\triangle, x_3=1/k-s\},&\\
\widetilde{Rr}^2_{k}:=\{(x_1,x_2,x_3)\in \partial \tilde{E}_{k}^3: (x_1,x_2)\in (E_k)^\triangle, x_3=s-1/k\}.&
\end{split}
\end{equation*}
Moreover,
\begin{equation}\label{equ1}
(\tilde{E}_{k}^3)^\triangle=(E_k)^\triangle\times [1/k-s,s-1/k].
\end{equation}
Let $a'_kb'_kc'_k\supset a_1b_1c_1$, $k=0,1,2,\ldots$,  be the open triangles the sides of which are parallel to the respective sides of $a_1b_1c_1$ and the distance between the sides of $a'_kb'_kc'_k$ and the respective sides of $a_1b_1c_1$ equals $1-1/k$. Then $a'_1b'_1c'_1= a_1b_1c_1$, $a'_{k+1}b'_{k+1}c'_{k+1}\supset a'_kb'_kc'_k$, and
\begin{equation}\label{theor5_0}
a'_0b'_0c'_0\supset a'_kb'_kc'_k\supset (E_k)^\triangle\supset (E_0)^\triangle,\,k=1,2,\ldots.
\end{equation}
Consider the triangles
\begin{equation*}
\begin{split}
Rl^2_{k}:=\{(x_1,x_2,x_3)\in \partial \tilde{E}_{k}^3: (x_1,x_2)\in a'_kb'_kc'_k, x_3=1/k-s\},&\\
Rr^2_{k}:=\{(x_1,x_2,x_3)\in \partial \tilde{E}_{k}^3: (x_1,x_2)\in a'_kb'_kc'_k, x_3=s-1/k\},& \,\,k=0,1,\ldots,
\end{split}
\end{equation*}
and some vector $\overrightarrow{a_3}$ generating an angle greater than  $0$ and less than $\dfrac{\pi}{2}$ with the positive direction of the axis $Ox_3$. This provides that two oblique prisms $Ll^3_{k}$, $Lr^3_{k}$ with respective bases $Rl^2_{k}$, $Rr^2_{k}$ and generatrices parallel to the vector $\overrightarrow{a_3}$ are such that $Ll^3_{0}\supset Ll^3_{k+1}\supset Ll^3_{k}$, $Lr^3_{0}\supset Lr^3_{k+1}\supset Lr^3_{k}$, $k=1,2,\ldots$ (see Figure~\ref{Fig11}~b)).

Remove  the closures of the prisms $Ll^3_{k}$, $Lr^3_{k}$ from the set $\tilde{E}_{k}^3$, $k=0,1,\ldots$ (see Figure~\ref{Fig11}~a)).
Then, considering (\ref{theor5_0}), the sets
$$
E^3_{k}:=\tilde{E}_k^3\setminus (\overline{Ll^3_{k}}\cup \overline{Lr^3_{k}}),\quad k=0,1,\ldots,
$$
 are weakly $1$-convex domains.
 Moreover, choose $s$ form (\ref{theor5}) large enough so that $\overline{Ll^3_{0}}\cap \overline{Lr^3_{0}}=\varnothing.$ Then
 \begin{equation}\label{equ2}
 (E^3_{k})^\triangle=(\tilde{E}_{k}^3)^\triangle\setminus (\overline{Ll^3_{k}}\cup \overline{Lr^3_{k}})\ne\varnothing,\quad k=0,1,\ldots.
 \end{equation}

Thus, the domains $E_{k}^3\subset \mathbb{R}^3$, $k=0,1,\ldots$,
belong to the class $\mathbf{WC^3_1}\setminus \mathbf{C^3_1}$. And the closure $\overline{E_{0}^3}$ of the set $E_{0}^3$ is approximated from the outside by the family of the domains $E_{k}^3$, $k=1,2,\ldots$ (see Figure~\ref{Fig11}~b)). Moreover, $\overline{E_{0}^3}$ is not $1$-convex by Lemma~\ref{lemm4}.  Thus, the closed and connected set $\overline{E_{0}^3}$ belongs to the class $\mathbf{WC^3_1}\setminus \mathbf{C^3_1}$.

\begin{figure}[h]
    \centering
   \includegraphics[width=14 cm]{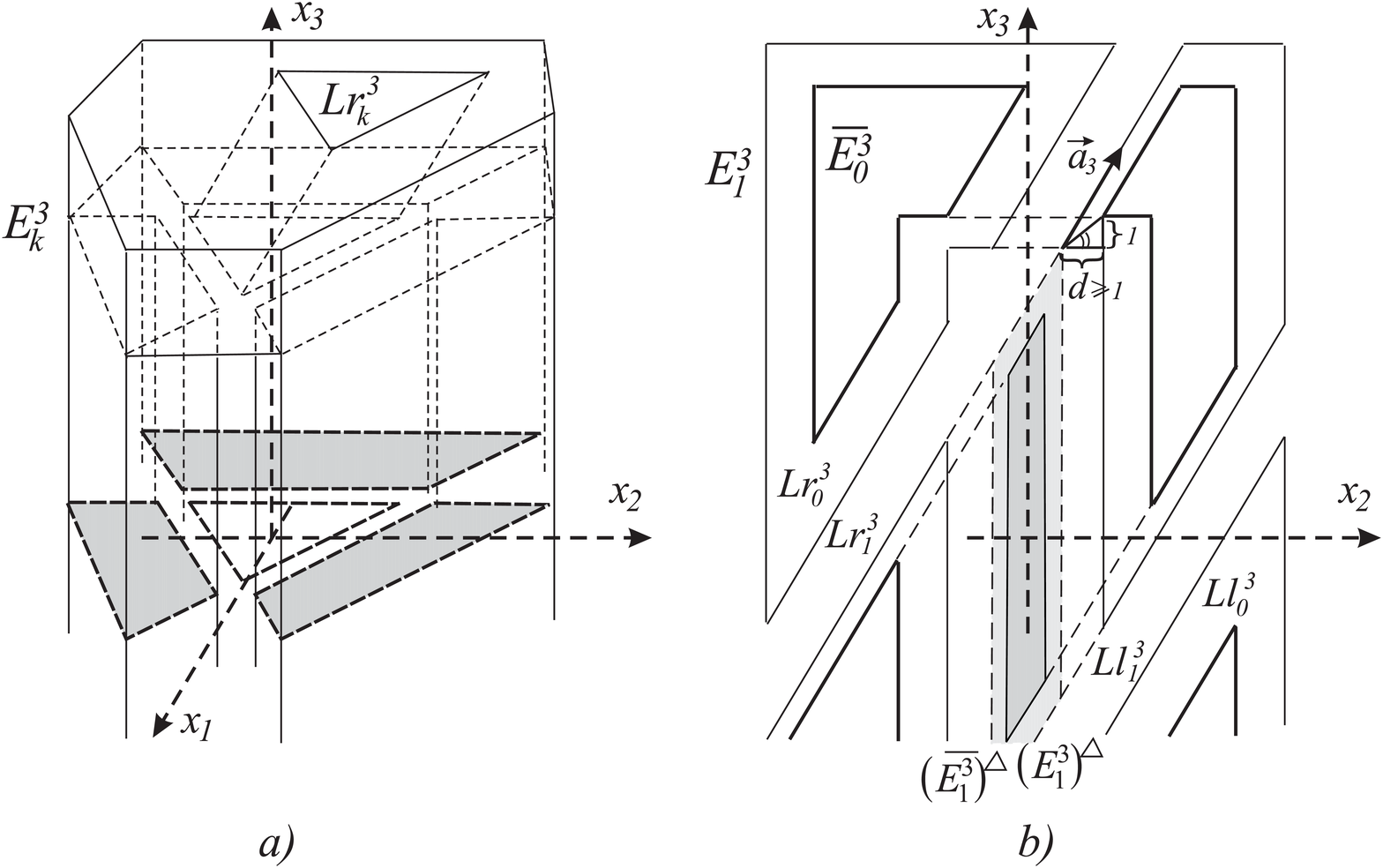}
    \caption{}\label{Fig11}
\end{figure}

Construct domains in the space
$\mathbb{R}^4$ of the class  $\mathbf{WC^4_1}\setminus \mathbf{C^4_1}$ approximating from the outside a closed connected set of the same class.

Consider the sets
$$
\tilde{\tilde{E}}_{k}^4:=E^3_{k}\times [1/k-s,s-1/k], \quad k=0,1,2,\ldots.
$$
Let  $P_{k}^3\subset\mathbb{R}^2$ be the convex hull of the set $E^3_{k}$, $k=0,1,2,\ldots$. Construct the following prisms:
\begin{equation*}
\begin{split}
Pl^4_{k}:=P_{k}^3\times \left[-1/k-1-s, 1/k-s\right],&\\
Pr^4_{k}:=P_{k}^3\times
\left[s-1/k,s+1+1/k\right],&\quad  k=0,1,\ldots.
\end{split}
\end{equation*}
Now consider the sets
$$
\tilde{E}_{k}^4:=\mathrm{Int}\,(Pl^4_{k}\cup \tilde{\tilde{E}}^4_{k}\cup Pr^4_{k}),\quad k=0,1,\ldots.
$$
They are $1$-convex with respect to any point of $\partial \tilde{E}_{k}^4$ except the points of  the sets
\begin{equation*}
\begin{split}
\widetilde{Rl}^3_{k}:=\{(x_1,x_2,x_3,x_4)\in \partial \tilde{E}_{k}^4: (x_1,x_2,x_3)\in (E_{k}^3)^\triangle, x_4=1/k-s\},&\\
\widetilde{Rr}^3_{k}:=\{(x_1,x_2,x_3,x_4)\in \partial \tilde{E}_{k}^4: (x_1,x_2,x_3)\in (E_{k}^3)^\triangle, x_4=s-1/k\}.&
\end{split}
\end{equation*}
Moreover,
$$
(\tilde{E}_{k}^4)^\triangle=(E_{k}^3)^\triangle\times [1/k-s,s-1/k],\quad k=0,1,\ldots.
$$
Construct the prisms
$$
L_{k}^3:=a'_kb'_kc'_k\times  [1/k-s,s-1/k], \quad k=0,1,\ldots.
$$
Then, considering (\ref{equ1}),  (\ref{theor5_0}), (\ref{equ2}),
\begin{equation}\label{theor5_1}
L_{k}^3\supset (\tilde{E}_{k}^3)^\triangle\supset (E^3_{k})^\triangle,\quad k=0,1,\ldots.
\end{equation}
Now consider the following sets:
\begin{equation*}
\begin{split}
Rl^3_{k}:=\{(x_1,x_2,x_3,x_4)\in \partial \tilde{E}_{k}^4: (x_1,x_2,x_3)\in L_{k}^3, x_4=1/k-s\},&\\
Rr^3_{k}:=\{(x_1,x_2,x_3,x_4)\in \partial \tilde{E}_{k}^4: (x_1,x_2,x_3)\in L_{k}^3, x_4=s-1/k\},&
\end{split}
\end{equation*}
$k=0,1,\ldots$. Since $L_{k+1}^3\supset L_{k}^3$, then $Rl^3_{k+1}\supset Rl^3_{k}$ and $Rr^3_{k+1}\supset Rr^3_{k}$. Moreover, considering (\ref{theor5_1}),
\begin{equation}\label{theor5_2}
Rl^3_{k}\supset \widetilde{Rl}^3_{k},\quad Rr^3_{k}\supset \widetilde{Rr}^3_{k}.
\end{equation}
 Consider some vector $\overrightarrow{a_4}$ generating an angle greater than  $0$ and less than $\dfrac{\pi}{2}$ with the positive direction of the axis $Ox_4$. This provides that two oblique prisms $Ll^4_{k}$, $Lr^4_{k}$ with respective bases $Rl^3_{k}$, $Rr^3_{k}$ and generatrices parallel to the vector $\overrightarrow{a_4}$ are such that $Ll^4_{0}\supset Ll^4_{k+1}\supset Ll^4_{k}$, $Lr^4_{0}\supset Lr^4_{k+1}\supset Lr^4_{k}$, $k=1,2,\ldots$.

Remove  the closures of the prisms $Ll^4_{k}$, $Lr^4_{k}$  from the set $\tilde{E}_{k}^4$, $k=0,1,\ldots$.
Then, considering (\ref{theor5_2}), the obtained sets
\begin{equation*}
E^4_{k}:=\tilde{E}_k^4\setminus (\overline{Ll^4_{k}}\cup \overline{Lr^4_{k}}),\quad k=0,1,\ldots,
\end{equation*}
 are weakly $1$-convex domains.
 Moreover, choose $s$ form (\ref{theor5}) large enough so that
 $\overline{Ll^4_{0}}\cap \overline{Lr^4_{0}}=\varnothing.$ Then
$$
 (E^4_{k})^\triangle=(\tilde{E}_{k}^4)^\triangle\setminus (\overline{Ll^4_{k}}\cup \overline{Lr^4_{k}})\ne\varnothing,\quad k=0,1,\ldots.
$$

Thus, the domains $E_{k}^4\subset \mathbb{R}^4$, $k=0,1,\ldots$,
belong to the class $\mathbf{WC^4_1}\setminus \mathbf{C^4_1}$. And the closure $\overline{E_{0}^4}$ of the set $E_{0}^4$ is approximated from the outside by the family of the domains $E_{k}^4$, $k=1,2,\ldots$. Moreover,  $\overline{E_{0}^4}$ is not $1$-convex by Lemma \ref{lemm4}.  Thus, the closed and connected set $\overline{E_{0}^4}$ belongs to the class $\mathbf{WC^4_1}\setminus \mathbf{C^4_1}$.

Extending the process of constructing the sets $E_k^n$, $k=1,2,\ldots$, and $\overline{E_{0}^n}$  to the spaces $\mathbb{R}^n$, $n>4$, using the sets $E_k^{n-1}$, $\overline{E_{0}^{n-1}}$ by the induction, we obtain domains and closed connected sets of the class $\mathbf{WC^n_1}\setminus \mathbf{C^n_1}$ for any $n\ge 3$. Then, by Lemma \ref{pro3}, the domains
$$
E_k^{n-m+1}\times\mathbb{R}^{m-1}\subset\mathbb{R}^{n},\quad  n\ge 3,\,\,1\le m<n-1,\,\, k=1,2,\ldots,
$$
and
the closed connected sets
$$
\overline{E_0^{n-m+1}}\times\mathbb{R}^{m-1}\subset\mathbb{R}^{n},\quad  n\ge 3,\,\,1\le m<n-1,
$$
belong to the class $\mathbf{WC^n_m}\setminus \mathbf{C^n_m}$. The theorem is proved.
\end{proof}


\bibliographystyle{pigc_plain}
\bibliography{biblio}

\end{document}